\documentclass[a4paper, 12pt,reqno]{amsart}
\usepackage{amssymb}
\usepackage{amsmath}

\newtheorem{thm}{Theorem}[section]

\newcommand {\bsig}{\boldsymbol{\sigma}}
\newcommand {\bnabla}{\boldsymbol{\nabla}}

\newcommand {\bgamma}{\boldsymbol{\gamma}}

\newcommand {\bx}{\mathbf x}
\newcommand {\by}{\mathbf y}
\newcommand {\bp}{\mathbf p}
\newcommand {\ba}{\mathbf a}

\newcommand{\R}{ \mathbb R}
\newcommand{\C}{ \mathbb C}

\makeatletter
   
          \@addtoreset{equation}{section}
\makeatother

\begin{document}

\baselineskip=1.2\baselineskip

\title{The Dirac-Hardy and Dirac-Sobolev
inequalities in $L^1$}

\author{Alexander Balinsky, W. Desmond Evans\\}
\author{Tomio Umeda${}^{\dag}$}


\address{
School of Mathematics \\
Cardiff University,
23 Senghennydd Road, Cardiff CF24  4AG, United Kingdom}
\email{balinskya@carfiff.ac.uk}

\address{
School of Mathematics \\
Cardiff University,
23 Senghennydd Road, Cardiff CF24  4AG, United Kingdom}
\email{evanswd@carfiff.ac.uk}

\address{Department of Mathematical Sciences,
University of Hyogo, Himeji 671-2201, Japan
}
\email{umeda@sci.u-hyogo.ac.jp}


\maketitle


\textbf{Abstract.} Dirac-Sobolev and Dirac-Hardy inequalities in
$L^1$ are established in which the $L^p$ spaces which feature in
the classical Sobolev and Hardy inequalities are replaced by weak
$L^p$ spaces. Counter examples to the analogues of the classical
inequalities are shown to be provided by zero modes for
appropriate Pauli operators constructed by Loss and Yau.

 \vspace{15pt}


 \vspace{15pt}
\textbf{The 2000 Mathematical Subject Classification:} Primary 35R45; Secondary 35Q40

\vspace{20pt}

\section{Introduction}  \label{sec:introduction}

Let $ \bsig = (\sigma_1, \sigma_2, \sigma_3)$ be the triple of $2
\times 2$ Pauli matrices
\begin{equation}  \label{eqn:PauliMatrices}
\sigma_1 =
\begin{pmatrix}
0&1 \\ 1& 0
\end{pmatrix}, \,\,\,
\sigma_2 =
\begin{pmatrix}
0& -i  \\ i&0
\end{pmatrix}, \,\,\,
\sigma_3 =
\begin{pmatrix}
1&0 \\ 0&-1
\end{pmatrix}
\end{equation}
and set
\[
\bp := -i \bnabla, \ \ \ \bsig \cdot \bp = -i \sum_{j=1}^3
\sigma_j \frac{\partial}{\partial x_j}.
\]
By the Dirac-Sobolev inequality we mean the following: that when
$1 \le p < 3,~ p^{*} = 3p/(3-p),$ and for all $f \in
C_0^{\infty}(\R^3,\C^2)$, the space of $\C^2$-valued functions
whose components lie in $C_0^{\infty}(\R^3)$,
\begin{equation}  \label{eqn:DiracSobolevIneq}
\left( \int_{\R^3} |f(\bx)|_{p^*}^{p^*} d \bx \right)^{1/p^*} \le C(p)
\left( \int_{\R^3} |(\bsig \cdot \bp) f(\bx)|_p^{p} d \bx
\right)^{1/p}
\end{equation}
where for $\ba = (a_1,a_2) \in \C^2,~ |\ba|_p^p = |a_1|^p +
|a_2|^p.$ It is shown by Ichinose and Sait\={o} in
\cite{IchinoseSaito} (see \lq\lq Note added in proof\rq\rq\ 
 at end
of paper) that for $1<p<\infty$, there are positive
constants $c_1(p), c_2(p)$ which are such that
\begin{equation} \label{eqn:IchiSait}
c_1(p) \int_{\R^3} |\bp f(\bx)|_p^{p} d \bx \le \int_{\R^3}
|(\bsig \cdot \bp)f(\bx)|_p^{p} d \bx \le c_2(p) \int_{R^3} |\bp
f(\bx)|_p^{p} d \bx,
\end{equation}
and hence for $1<p<3$, (\ref{eqn:DiracSobolevIneq}) is a
consequence of the Sobolev inequality
\begin{equation}  \label{eqn:SobolevIneq}
\left( \int_{\R^3} |f(\bx)|_{p^*}^{p^*} d \bx \right)^{1/p^*} \le
\tilde{C}(p) \left( \int_{\R^3} |\bp f(\bx)|_p^{p} d \bx
\right)^{1/p}.
\end{equation}

 On defining the
Dirac-Sobolev space $H_{D,0}^{1,p}(\R^3,\C^2)$ to be the
completion of $C_0^{\infty}(\R^3,\C^2)$ with respect to the norm
\[
\|f\|_{D,1,p} := \left \{ \int_{\R^3} [|f(\bx)|^p_p + |(\bsig
\cdot \bp) f(\bx)|^p_p ]d \bx \right \}^{1/p},
\]
(\ref{eqn:IchiSait}) proves that $H^{1,p}_{D,0}(\R^3,\C^2)$ is
isomorphic to  $H^{1,p}_0(\R^3,\C^2)$ if $1<p<\infty$,
where  $H^{1,p}_0(\R^3,\C^2)$ denotes the Sobolev space defined
to be the completion of
 $C_0^{\infty}(\R^3,\C^2)$ with respect to the norm
\[
\|f\|_{S,1,p} := \left \{ \int_{\R^3} [|f(\bx)|^p_p +
| \bp f(\bx)|^p_p ]d \bx \right \}^{1/p}.
\]
However, as $p \rightarrow 1,~c_1(p) \rightarrow 0 $ and so
(\ref{eqn:IchiSait}) only implies that $H^{1,1}_0(\R^3,\C^2)$ is
continuously embedded in $H^{1,1}_{D,0}(\R^3,\C^2).$ In fact
Ichinose and Sait\={o} go on to prove that the embedding
$H^{1,1}_0(\R^3,\C^2) \hookrightarrow H^{1,1}_{D,0}(\R^3,\C^2)$ is
indeed strict. Hence, in the case $p=1,$
(\ref{eqn:DiracSobolevIneq}) is not a consequence of the analogous
Sobolev inequality. We prove that the $p=1$ case of
(\ref{eqn:DiracSobolevIneq}) is untrue. We demonstrate this with a
function used by Loss and Yau in \cite{LossYau} to prove the
existence of zero modes of a Pauli operator $ \{\bsig \cdot (\bp +
{\bf{A}})\}^2$ (or equivalently, of the Weyl-Dirac operator $
\bsig \cdot (\bp + {\bf{A}})$ ) with some appropriate magnetic
potential $ {\bf{A}}$. A result of Sait\={o} and Umeda in
\cite{SaitoUmeda-1} on the growth properties of zero modes of
Pauli operators indicates that zero modes have quite generally the
properties we need of the counter-example. We prove in Theorem
\ref{thm:WeakDiracSobolevIneq-1}
\begin{equation}\label{eqn:DiracSobolevIneq-2}
    \|f\|_{3/2,\infty} \le C_1 \int_{\R^3} |(\bsig \cdot
    \bp)f(\bx)| d \bx,
\end{equation}
where $|\cdot| = |\cdot|_1$ and for any $ q > 0,$
\begin{equation}\label{weakLi}
\|f\|_{q,\infty}^q :=  \sup_{t > 0} ~t^q \mu (\{\bx \in \R^3:
|f(\bx)|>t \}),
\end{equation}
$\mu$ denoting Lebesgue measure. We recall that $\|f\|_{q,\infty}
< \infty$ if and only if $f$ belongs to the weak-$L^q$ space
$L^{q,\infty}(\R^3,\C^2).$ Moreover, $ \|\cdot\|_{q,\infty}$ is
not a norm on $L^{q,\infty}(\R^3,\C^2)$ but for $q>1$ it is
equivalent to a norm; see \cite{EdmundsEvans}, Section 3.4.

Analogous questions arise for the Dirac-Hardy inequality
\begin{equation}  \label{eqn:DiracHardyIneq-1}
\int_{\R^3} \frac{|f(\bx)|_p^p}{|\bx|^p} d\bx \le C(p)
\int_{{\mathbb R}^3} |(\bsig \cdot \bp) f(\bx)|_p^p d \bx
\end{equation}
and similar answers are obtained. The inequality is true for $1< p
<\infty$ by $ (\ref{eqn:IchiSait})$, but not for $p=1$ in which case we
prove that
\begin{equation}  \label{eqn:DiracHardyIneq-2}
\big\Vert  |f|/|\cdot| \big\Vert_{1, \infty} \le C_2 \int_{\R^3}
|(\bsig \cdot \bp) f(\bx)|d \bx .
\end{equation}

The plan of the paper is as follows. In Section 2 we shall prove
the results concerning the Dirac-Sobolev and Dirac-Hardy
inequalities discussed above. We shall give estimates of the
optimal constant  $C(p)$  in the Dirac-Sobolev inequality
(\ref{eqn:DiracSobolevIneq}) for $1<p<3$ in Section
\ref{sec:OptimalConstant} and show that $C(p) \to\infty$ as
$p\downarrow 1$. In order to check if the results in Section 2 are
dimension related, we investigate higher dimensional analogues in
Section \ref{sec:DiracSobolevHardy-m}. A weak H\"{o}lder-type
inequality is given in an Appendix.


\section{The weak Dirac-Sobolev and Dirac-Hardy inequalities}
\label{sec:DiracSobolevHardy}

To show that the inequality (\ref{eqn:DiracSobolevIneq}) does not
hold, we shall prove that a counter-example is provided by a zero
mode for an appropriate Pauli (or Weyl-Dirac) operator constructed
by Loss-Yau in \cite{LossYau}. This is the $\C^2$-valued function
\begin{equation}  \label{eqn:LossYauZeroMode}
\psi(\mathbf x) = \frac{1}{(1+r^2)^{3/2}} (I + i \mathbf x \cdot
\bsig)
\begin{pmatrix}
1  \\  0
\end{pmatrix},
\quad r=|\mathbf x|,
\end{equation}
where $I$ is the $2\times 2$ identity matrix.
In view of  the anti-commutation
relation $\sigma_j \sigma_k + \sigma_k \sigma_j = 2 \delta_{jk}I$,
it follows that
\begin{equation}  \label{eqn:opt-2}
|\psi(\mathbf x) | = \frac{1}{1 + r^2}.
\end{equation}
Also, $\psi$ satisfies  the Loss-Yau equation
\begin{equation}    \label{eqn:LossYauZeroMode-1}
(\bsig \cdot \mathbf p) \psi (\mathbf x) = \frac{3}{1 + r^2}
\psi(\mathbf x).
\end{equation}
Let $\chi_n \in C_0^{\infty}(\mathbb R)$ be
such that
\begin{equation}  \label{eqn:cutoff-n}
\chi_n(r) =
\begin{cases}
1,  &  r \le n  \\
0,  &   r \ge n + 2,
\end{cases}
\quad |\chi_n^{\prime}(r)| \le 1.
\end{equation}
Then  $\psi_n:= \chi_n \psi
\in C_0^{\infty}({\mathbb R}^3, {\mathbb C}^2)$
and we see that
\begin{gather}   \label{eqn:LY-1}
\begin{split}
\Vert
  (\bsig \cdot \mathbf p) \psi_n
 \Vert_{L^1({\mathbb R}^3, {\mathbb C}^2)}
&= \Vert \chi_n(\bsig \cdot \mathbf p)  \psi -i \, \chi_n^{\prime}
(\bsig \cdot \frac{\mathbf x}{r}) \psi
\Vert_{L^1({\mathbb R}^3, {\mathbb C}^2)}  \\
&\le
4\pi \Big( \int_0^{n+2} \frac{3}{\, 1+r^2 \,} \, dr
  + \int_n^{n+2}   \, dr  \Big)                   \\
&\le C_0,
\end{split}
\end{gather}
for some constant positive $C_0$, independent of $n$.

Now suppose that the case $p=1$ of the inequality
(\ref{eqn:DiracSobolevIneq}) is true. Then it would follow from
(\ref{eqn:LY-1}) that
\begin{gather}   \label{eqn:LY-2}
\begin{split}
C_0
&\ge
\Vert
 \psi_n
 \Vert_{L^{3/2}({\mathbb R}^3, {\mathbb C}^2)}   \\
&\ge
\left(
 \int_{|\mathbf x|\le n} |\psi(\mathbf x)|^{3/2} \, d\mathbf x
 \right)^{2/3}\\
&\ge
\mbox{const.} \, (\log n)^{2/3}.
\end{split}
\end{gather}
and hence a contradiction.

The properties of the zero mode  $\psi$, defined by
(\ref{eqn:LossYauZeroMode}), which lead to the inequality
(\ref{eqn:DiracSobolevIneq}) being contradicted when $p=1$ are
that $(\bsig \cdot \mathbf p) \psi \in L^1({\mathbb R}^3, {\mathbb
C}^2)$ and $\psi(\mathbf x) \asymp r^{-2}$ at infinity ({\it
i.e.}, $r^2\psi(\mathbf x)$ goes to a constant vector in ${\mathbb
C}^2$ as $r \to \infty$). It was shown in Sait\={o}-Umeda
\cite{SaitoUmeda-1} that these two properties are satisfied by the
zero modes of any Weyl-Dirac operator
\begin{equation}  \label{eqn:WeylDiracOp}
{\mathbb D}_A=\bsig \cdot \big( \mathbf p
 + \mathbf A(\mathbf x)  \big)
\end{equation}
whose magnetic potential
$\mathbf A  \; =(A_1, \, A_2, \, A_3)$
is such that
\begin{equation}
A_j \; \mbox{ is measurable}, \quad
  |A_j(\mathbf x)| \le C (1 + r)^{-\rho}, \;\;\; \rho >1
\end{equation}
for $j=1$, $2$, $3$.

As was mentioned in the Introduction, what is true is the following

\begin{thm}  \label{thm:WeakDiracSobolevIneq-1}
There exists a positive constant $C_1$ such that
\begin{equation}   \label{eqn:DiracSobolevIneq-3}
\Vert f \Vert_{L^{3/2, \infty}({\mathbb R}^3, {\mathbb C}^2)}
\le
C_1 \Vert
 (\sigma \cdot \mathbf p) f
 \Vert_{L^1({\mathbb R}^3, {\mathbb C}^2)}
\end{equation}
for all $f \in C_0^{\infty}({\mathbb R}^3, {\mathbb C}^2)$.
\end{thm}

\begin{proof}
Let $g=  (\sigma \cdot \mathbf p) f$. Since $(\bsig \cdot \bp)^2 =
-\Delta$ and the fundamental solution of $-\Delta $ in $\R^3$ is
convolution with $(1/4 \pi |\cdot|)$, it follows that $(\bsig
\cdot \bp)$ has fundamental solution with kernel $ (\bsig \cdot
\bp)(1/4 \pi |\cdot|)$ and hence that
\begin{eqnarray} \label{eqn:Riesz-1}
f(\mathbf x)&=&  \frac{-i}{4\pi} \int_{{\mathbb R}^3} [( \bsig
\cdot \bnabla){|\mathbf x-\mathbf y|^{-1}}]
 g(\mathbf y) \, d\mathbf y \nonumber \\
 &=&  \frac{i}{4\pi} \int_{{\mathbb R}^3}
\frac{ \bsig \cdot (\mathbf x - \mathbf y)}{|\mathbf x-\mathbf
y|^3} g(\by) d \by.
\end{eqnarray}
Note that this also follows from the more general result in
Sait\={o}-Umeda \cite[Theorem 4.2]{SaitoUmeda-2}. Consequently
\begin{equation}  \label{eqn:Riesz-2}
|f(\mathbf x)|
\le
\displaystyle{
\frac{1}{4\pi}
\int_{{\mathbb R}^3}
\frac{1}{|\mathbf x-\mathbf y|^2}
 |g(\mathbf y)| \, d\mathbf y=:
\frac{1}{4\pi}I_1(|g|)(\mathbf x)},
\end{equation}
where $I_1(|g|)$ is the 3-dimensional Riesz potential of $|g|$;
see Edmunds and Evans \cite[Section 3.5]{EdmundsEvans} for the
terminology and properties we need. In view of \cite[Remark
3.5.7(i)]{EdmundsEvans}, we see that the Riesz potential $I_1$ is
of weak type $(1, \, 3/2; 3, \, \infty)$. In particular, $I_1$ is
of weak type $(1, \, 3/2)$ (cf. \cite[Theorem
3.5.13]{EdmundsEvans},  \cite[Theorem 1, pp.119 - 120]{Stein}), which means that there exists a positive
constant $C$ such that for all $u \in L^1({\mathbb R}^3)$
\begin{equation}
\Vert I_1(u) \Vert_{L^{3/2, \infty}({\mathbb R}^3)}
\le C
\Vert u \Vert_{L^1({\mathbb R}^3)}.
\end{equation}
The inequality (\ref{eqn:DiracSobolevIneq-3}) follows.
\end{proof}

It is evident that the two properties of the zero mode $\psi$
defined by  (\ref{eqn:LossYauZeroMode}) also leads to a
contradiction of the inequality (\ref{eqn:DiracHardyIneq-1}). What
is now true is the following:

\begin{thm}   \label{thm:WeakDiracHardyIneq-1}
For all $f \in C_0^{\infty}({\mathbb R}^3, {\mathbb C}^2)$
\begin{equation}  \label{eqn:DiracHardyIneq-3}
\Vert
 f/|\cdot|
 \Vert_{L^{1, \infty}({\mathbb R}^3, \,{\mathbb C}^2)} \le C_2 \Vert (\bsig \cdot \mathbf p) f \Vert_{L^1({\mathbb R}^3,
\,{\mathbb C}^2)} ,
\end{equation}
where $C_{2} \le (9 \pi)^{1/3} C_1$ and
 $C_1$ is the optimal constant in {\rm(\ref{eqn:DiracSobolevIneq-3})}.
\end{thm}

\begin{proof}
On applying the weak H\"older inequality in the Appendix with
$p=3/2$ and $q=3$, and noting that $\Vert 1/|\cdot| \Vert_{3,
\infty} = (4\pi /3)^{1/3}$, we get
\begin{equation}  \label{eqn:DiracHardyIneq-4}
\Vert
 f/|\cdot|
 \Vert_{1, \infty}
\le 3^{2/3} \pi^{1/3}
\Vert f \Vert_{3/2, \infty}.
\end{equation}
Hence the theorem follows from (\ref{eqn:DiracSobolevIneq-3}).
\end{proof}


\section{Estimate of the optimal constants}
\label{sec:OptimalConstant}

In this section, we estimate the optimal constant
$C(p)$ in the inequality (\ref{eqn:DiracSobolevIneq})
for $1 < p <3$, and show that
$C(p) \to \infty$ as $p \downarrow 1$.

Let $\psi$ be the Loss-Yau zero mode
 defined by (\ref{eqn:LossYauZeroMode}). It does not lie in
 $C_0^{\infty}(\R^3,\C^2)$ but is in $H^{1,p}_{D,0}(\R^3,\C^2).$
Hence the optimal constant $C(p)$ must satisfy the inequality
\begin{equation}  \label{eqn:opt-1}
C(p) \ge \Vert \psi \Vert_{L^{p^*}({\mathbb R}^3, \,{\mathbb
C}^2)} \big/ \Vert (\bsig \cdot \mathbf p)\psi \Vert_{L^p
({\mathbb R}^3, \,{\mathbb C}^2)},
\end{equation}
where $p^*= 3p/(3-p)$. On passing to polar coordinates, we have
\begin{gather}  \label{eqn:opt-3}
\begin{split}
\Vert
\psi
\Vert_{{p^*}}^{p^*}
&=
4 \pi \int_0^{\infty}  (1 + r^2)^{-p^*} r^2 dr  \\
&\ge
4\pi \big\{
\int_0^1 2^{-p^*} r^2 \, dr
 +
  \int_1^{\infty} (2r^2)^{-p^*} r^2  dr \big\}   \\
&=
4\pi\, 2^{-p^*}\, 3^{-2} \frac{2p}{p-1}.
\end{split}
\end{gather}
On the other hand, by (\ref{eqn:LossYauZeroMode-1}), we see that
\begin{gather}  \label{eqn:opt-4}
\begin{split}
\Vert (\bsig \cdot \mathbf p)\psi \Vert_p^p &=
\int_{{\mathbb R}^3} \frac{3^p}{(1 + r^2)^{2p}} \, dx   \\
&=
4 \pi \, 3^p \int_0^{\infty}  (1 + r^2)^{-2p} r^2 dr  \\
&\le
4\pi \, 3^p \big\{
\int_0^1  r^2 \, dr
 +
  \int_1^{\infty} r^{-4p+2}  dr \big\}   \\
&=
\pi \, 2^4  3^{p-1}  \frac{p}{4p -3}.
\end{split}
\end{gather}
Combining (\ref{eqn:opt-1}) with
  (\ref{eqn:opt-3}) and (\ref{eqn:opt-4}),
we obtain
\begin{equation}  \label{eqn:opt-5}
C(p) \ge
\pi^{-1/3} \, 2^{-2 -1/p} \, 3^{-1/3 -1/p} \,
\frac{\, p^{-1/3}(4p-3)^{1/p} \,}{ (p-1)^{1/p - 1/3} }.
\end{equation}
It is evident that the right hand side
of (\ref{eqn:opt-5}) goes to $\infty$
as $p \downarrow 1$.

We recall that for $p>1,$ the optimal constant $\tilde{C}(p)$ in
the Sobolev inequality (\ref{eqn:SobolevIneq})  is
\[
\tilde{C}(p) =
\pi^{-1/2}3^{-1/p}\left(\frac{p-1}{3-p}\right)^{(p-1)/p}\left
\{\frac{\Gamma(5/2)\Gamma(3)}{\Gamma(3/p)\Gamma(4-3/p)}\right
\}^{1/3}
\]
which tends to $\tilde{C}(1)$, the optimal constant in the case
$p=1,$ as $p \rightarrow 1.$


\section{The weak Dirac-Sobolev and weak Dirac-Hardy inequalities\\
 in $m$ dimensions}
\label{sec:DiracSobolevHardy-m}

Let $\gamma_1$, $\cdots$, $\gamma_m$ be Hermitian $ \ell \times
\ell$ matrices satisfying the anti-commutation relations
\begin{equation}  \label{eqn:gammamatrices-1}
\gamma_j \gamma_k + \gamma_k \gamma_j = 2 \delta_{jk} I,
\end{equation}
where $I$ denotes the $\ell \times \ell$ identity matrix. For
example, we can take $\ell=2^{m-2}$ and construct the matrices by
the following iterative procedure. To indicate the dependence on
$m$, write the matrices as $\gamma_1^{(m)}$, $\cdots$,
$\gamma_m^{(m)}$. For $m=3$, $\ell=2$ 
and  they are given by the Pauli matrices
in (\ref{eqn:PauliMatrices}). Given matrices $\gamma_1^{(m)}$,
$\cdots$, $\gamma_m^{(m)}$ we define
\begin{equation}
\gamma_j^{(m+1)} =
\begin{pmatrix}
0& \gamma_j^{(m)}  \\ \gamma_j^{(m)} & 0
\end{pmatrix},  \,\; j=1, \, \cdots, \, m,
\quad
\gamma_{m+1}^{(m+1)} =
\begin{pmatrix}
I & 0  \\
0  & -I
\end{pmatrix}.
\end{equation}

The $m$-dimensional analogue of the inequality
(\ref{eqn:DiracSobolevIneq}) for $p=1$  is
\begin{equation}  \label{eqn:DiracSobolevIneq-1-m}
\left( \int_{{\mathbb R}^m} |f(\mathbf x)|^{m/(m-1)} d \mathbf x
\right)^{(m-1)/m} \le C \int_{{\mathbb R}^m} |(\bgamma \cdot
\mathbf p) f(\mathbf x)| d \mathbf x
\end{equation}
for $f \in C^{\infty}_0({\mathbb R}^m, {\mathbb C}^{\ell})$,
where
\begin{equation*}
\bgamma \cdot \bp = -i\sum_{j=1}^m \gamma_j
\frac{\partial}{\partial x_j} , \quad \bp= -i \bnabla.
\end{equation*}
To show that (\ref{eqn:DiracSobolevIneq-1-m}) does not hold we
introduce an $m$-dimensional analogue of the Loss-Yau zero mode,
namely
\begin{equation}  \label{eqn:LossYauZeroMode-m}
\psi(\mathbf x) = \frac{1}{(1+r^2)^{m/2}} (I + i \mathbf x \cdot
\bgamma) \phi_0, \quad r=|\mathbf x|,
\end{equation}
where $\phi_0 ={}^t(1, \, 0, \cdots,\, 0) \in {\mathbb C}^{\ell}$.
It follows from the anti-commutation relations
(\ref{eqn:gammamatrices-1}) that
\begin{equation}  \label{eqn:opt-2-m}
|\psi(\mathbf x) | = \frac{1}{(1 + r^2)^{(m-1)/2}} \, ,
\end{equation}
and that $\psi$ satisfies the
$m$-dimensional analogue of the Loss-Yau
equation (\ref{eqn:LossYauZeroMode-1}),
namely,
\begin{equation}    \label{eqn:LossYauZeroMode-1-m}
(\bgamma \cdot \mathbf p) \psi (\mathbf x) = \frac{m}{1 + r^2}
\psi(\mathbf x).
\end{equation}
Let $\chi_n \in C_0^{\infty}(\mathbb R)$ be
the same function
as in (\ref{eqn:cutoff-n}), and put
 $\psi_n:= \chi_n \psi \in
C_0^{\infty}({\mathbb R}^m, {\mathbb C}^{\ell})$. As in to
(\ref{eqn:LY-1}), we see that
\begin{gather}   \label{eqn:LY-1-m}
\begin{split}
\Vert
  (\bgamma \cdot \mathbf p) \psi_n
 \Vert_{L^1({\mathbb R}^m, {\mathbb C}^{\ell}) }
&= \Vert \chi_n(\bgamma \cdot \mathbf p)  \psi -i \,\chi_n^{\prime}
(\bgamma \cdot \frac{\mathbf x}{r}) \psi
\Vert_{ L^1({\mathbb R}^m, {\mathbb C}^{\ell}) }  \\
&\le
S_m \Big( \int_0^{n+2} \frac{m}{\, 1+r^2 \,} \, dr
 +  \int_n^{n+2}   \, dr  \Big) \\
\noalign{\vskip 4pt}
&\le C_0,
\end{split}
\end{gather}
for some positive constant  $C_0$, independent of $n$.
Here $S_m$ is the surface area of 
the unit sphere in ${\mathbb R}^m$. If the
inequality (\ref{eqn:DiracSobolevIneq-1-m}) is true then it would
follow from (\ref{eqn:DiracSobolevIneq-1-m}) and
(\ref{eqn:LY-1-m}) that
\begin{gather}   \label{eqn:LY-2-m}
\begin{split}
C_0
&\ge
\Vert
 \psi_n
 \Vert_{L^{m/(m-1)}({\mathbb R}^m, {\mathbb C}^{\ell})}   \\
&\ge
\left(
 \int_{|\mathbf x|\le n}
|\psi(\mathbf x)|^{m/(m-1)} \, d\mathbf x
 \right)^{(m-1)/m}\\
&\ge
\mbox{const.} \, (\log n)^{(m-1)/m}.
\end{split}
\end{gather}
which is a contradiction. Therefore the inequality
(\ref{eqn:DiracSobolevIneq-1-m}) does not hold. Instead, what is
true is the following inequality.

\begin{thm}  \label{thm:WeakDiracSobolevIneq-1-m}
There exists a positive constant $C_{1,m}$ such that
\begin{equation}   \label{eqn:DiracSobolevIneq-3-m}
\Vert  f
\Vert_{L^{m/(m-1), \infty}({\mathbb R}^m, {\mathbb C}^{\ell})}
\le
C_{1,m}  \,
\Vert
 (\bgamma \cdot \mathbf p) f
 \Vert_{L^1({\mathbb R}^m, {\mathbb C}^{\ell})}
\end{equation}
for all $f \in C_0^{\infty}({\mathbb R}^m, {\mathbb C}^{\ell})$.
\end{thm}
\begin{proof}

Let $f \in C_0^{\infty}({\mathbb R}^m, {\mathbb C}^{\ell})$, and
define $g= (\bgamma \cdot \mathbf p)f$. Since $ (\bgamma \cdot
\bp)^2 = -\Delta I,$ we have that $(-\Delta)f = (\bgamma \cdot
\bp)g.$ By Stein \cite[p.118, (7)]{Stein},
\begin{equation}
J_2(-\Delta) u = u,  \quad u \in C_0^{\infty}({\mathbb R}^m, \mathbb C),
\end{equation}
where
\begin{equation}
J_2(u) = \frac{\Gamma((m-2)/2)}{4 \pi^{m/2}}I_2(u),
\quad
I_2(u)(\mathbf x) = \int_{{\mathbb R}^m}
\frac{1\;\;\;}{\, |\mathbf x-\mathbf y|^{m-2} \,} \, u (\mathbf y) \,
d \mathbf y.
\end{equation}
It follows that
\begin{gather}
\begin{split}
f(\bx)
& = J_2(-\Delta) f(\bx)   \\
\noalign{\vskip 4pt}
&=
\frac{\Gamma((m-2)/2)}{4 \pi^{m/2}}
 \int_{\R^m}
\frac{1\;\;\;}{\, |\bx-\by|^{m-2} \,} \, (\bgamma \cdot \bp)
g(\by) \, d \by.
\end{split}
\end{gather}
On integration by parts, this yields
\begin{equation}  \label{eqn:lemRiesz-m}
f(\mathbf x) = \frac{\Gamma((m-2)/2)}{4 \, \pi^{m/2}} \int_{\R^m}
\frac{\, i \, \bgamma \cdot (\bx-\by)}{|\bx-\by|^m} \, g(\by) \,
d\by.
\end{equation}
Then it follows that
\begin{gather}  \label{eqn:Riesz-2-m}
\begin{split}
|f(\mathbf x)|
&\le
\frac{\Gamma((m-2)/2)}{4 \pi^{m/2}}
\int_{{\mathbb R}^m}
\frac{1}{|\mathbf x-\mathbf y|^{m-1} }
 |g(\mathbf y)| \, d\mathbf y   \\
&=
\frac{\Gamma((m-2)/2)}{4 \pi^{m/2}}
\frac{2 \pi^{(m+2)/2}}{\Gamma((m-1)/2)}
\,
I_1(|g|)(\mathbf x).
\end{split}
\end{gather}
Here $I_1(|g|)$ is the $m$-dimensional Riesz potential of $|g|$; see
\cite[Section 3.5]{EdmundsEvans}.
In view of \cite[Remark 3.5.7(i)]{EdmundsEvans},
we see that the Riesz potential $I_1$ is
of weak type $(1, \, m/(m-1); m, \, \infty)$,
in particular,  of
weak type $(1, \, m/(m-1) )$
(cf.
\cite[Theorem 3.5.13]{EdmundsEvans},
 \cite[Theorem 1, pp.119 - 120]{Stein}),
which means that
there exists a positive constant $C$ such that
for all $u \in L^1({\mathbb R}^m)$
\begin{equation}
\Vert I_1(u) \Vert_{L^{m/(m-1), \infty}({\mathbb R}^m)}
\le C
\Vert u \Vert_{L^1({\mathbb R}^m)}.
\end{equation}
The inequality (\ref{eqn:DiracSobolevIneq-3-m}) follows.
\end{proof}

\vspace{6pt}

The $m$-dimensional Hardy inequality for $L^1$ is
\begin{equation}  \label{eqn:Hardy-m}
\int_{\R^m}
 \frac{|u(\bx)|}{|\bx|} d \bx \le (m-1)^{-1}\int_{\R^m} |\bp ~ u(\bx)| d \bx  ,\ \ \  u \in C^{\infty}_0(\R^m).
\end{equation}
It can be proved as in the 3-dimensional case that its natural
analogue
\begin{equation}  \label{eqn:DiracHardyIneq-1-m}
\int_{\R^m}
  \frac{|f(\bx)|}{|\bx|} d \bx \le C \int_{\R^m} |(\bgamma \cdot \bp) f(\bx)|d \bx  ,\ \ \  f \in C^{\infty}_0(\R^m, \C^{\ell})
\end{equation}
is not true.

\begin{thm}   \label{thm:WeakDiracHardyIneq-1-m}
For all $f \in C_0^{\infty}(\R^m, \C^{\ell})$
\begin{equation}  \label{eqn:DiracHardyIneq-3-m}
\Vert
 f/|\cdot|
 \Vert_{L^{1, \infty}(\R^m, \C^{\ell})} \le
C_{2,m} \Vert (\bgamma \cdot \bp) f \Vert_{L^1(\R^m,
\,\C^{\ell})},
\end{equation}
where
\begin{equation*}
C_{2,m} \le C_{1,m} \frac{\pi^{1/2} \, m}{\Gamma \big( (m+2)/2
\big)^{1/m} (m-1)^{1-1/m}} ,
\end{equation*}
and $C_{1,m}$ is the optimal constant  in 
Theorem {\rm\ref{thm:WeakDiracSobolevIneq-1-m}}.
\end{thm}

\begin{proof}
It is easy to see that
$\Vert 1/|\cdot| \Vert_{m, \infty} = (\omega_m)^{1/m}$,
where
$\omega_m$ denotes the volume of the $m$-dimensional unit ball,
and is given by
\begin{equation*}
\omega_m=\frac{\pi^{m/2}}{\Gamma((m+2)/2)}.
\end{equation*}
On applying the weak H\"older inequality in the Appendix with
$p=m/(m-1)$ and $q=m$, we get
\begin{equation}  \label{eqn:DiracHardyIneq-4-m}
\Vert
 f/|\cdot|
 \Vert_{1, \infty}
\le
\big( (m-1)^{1/m}  + (m-1)^{-(m-1)/m} \big) \, \omega_m^{1/m}
\Vert f \Vert_{m/(m-1), \infty}.
\end{equation}
The theorem follows on combining this inequality with
(\ref{eqn:DiracSobolevIneq-3-m}).
\end{proof}


\section{Appendix}
The proofs of Theorems \ref{thm:WeakDiracHardyIneq-1}
and \ref{thm:WeakDiracHardyIneq-1-m} are 
consequences of the following H\"older - type inequality in weak
$L^p$ spaces, which we have been unable to find in the literature.

\begin{thm}[\textbf{Weak H\"older inequality}]   \label{thm:WHolder}
Let  $p>1$, $q>1$ and $p^{-1}+ q^{-1}=1$. If
$f \in L^{p, \infty}({\mathbb R}^d)$
and $g \in L^{q, \infty}({\mathbb R}^d)$,
then $fg \in L^{1, \infty}$ and
\begin{equation} \label{eqn:WHI}
\Vert fg \Vert_{1, \infty}
\le
\Big( (q/p)^{1/q} + (p/q)^{1/p} \Big)
\Vert f \Vert_{p, \infty}\Vert g \Vert_{q, \infty}.
\end{equation}
\end{thm}
\begin{proof}

Let $\varepsilon > 0$ be arbitrary, and set
\begin{align*}
A &=\{  {\bf x} \in {\mathbb R}^d \, : \, \varepsilon |f({\bf
x})|>t^{1/p} \, \} \\
B &=\{ {\bf x}  \in {\mathbb R}^d\, : \, \frac{1}{\varepsilon}
|g({\bf x})| > t^{1/q} \, \}  \\
E &=\{ {\bf x}  \in {\mathbb R}^d\, : \, |f({\bf x})
g({\bf x})| > t \, \}.
\end{align*}
Since
\begin{equation}
|f({\bf x}) g({\bf x})| \le
p^{-1} \big( \varepsilon |f({\bf x})|\big)^p
+
q^{-1} \big( \frac{1}{\varepsilon} |g({\bf x})|\big)^q,
\end{equation}
we have
\begin{equation}
E \subset A \cup B,
\end{equation}
which implies that
\begin{equation}  \label{eqn:abc-1}
t \mu (E)  \le
 t  \mu \big( \{ {\bf x} \, : \, \varepsilon |f({\bf x})|>t^{1/p} \, \} \big)
  + t \mu \big(\{ {\bf x} \, : \, \frac{1}{\varepsilon} |g({\bf x})| > t^{1/q} \, \} \big).
\end{equation}
With
\begin{equation}
s:= \frac{t^{1/p}}{\varepsilon}, \quad r:= \varepsilon t^{1/q},
\end{equation}
it follows from (\ref{eqn:abc-1}) that
\begin{align}
t &\mu \big( \{ {\bf x} \, : \, |f({\bf x}) g({\bf x})|>t \, \} \big)  \nonumber \\
&\le
\varepsilon^p s^p
 \mu \big( \{ {\bf x} \, : \,  |f({\bf x})|>s \, \} \big)
 +
\varepsilon^{-q} r^q
\mu \big( \{ {\bf x} \, : \, |g({\bf x})|>r \, \} \big)  \nonumber \\
&\le \varepsilon^p \Vert f \Vert^p_{p, \infty} + \varepsilon^{-q}
\Vert g \Vert^q_{q, \infty}.  \label{eqn:abc-2}
\end{align}
The minimum value of this is the expression on the right-hand side
of (\ref{eqn:WHI}), this being attained when $ \varepsilon = (q
\|g\|^q_{q,\infty}/p \|f\|^p_{p,\infty})^{1/pq}$.
\end{proof}

\subsection*{Acknowledgements}
TU is supported by 
  Grant-in-Aid for Scientific Research (C) 
    No.  21540193, 
  Japan Society for the Promotion of Science.
  AB and WDE are grateful for the support and hospitality of the 
  University of Hyogo in March when this work was 
  completed.


\vspace{15pt}
{\small 
\textbf{Note added in proof.} 
We are grateful to the
 referee for the comment that the first 
 inequality  in (\ref{eqn:IchiSait})
 is not 
a direct consequence of Ichinose and Sait\={o} 
\cite[Theorem 1.3(ii)]{IchinoseSaito},
 but can be established as follows.
 Since $(\sigma \cdot \bp)^2=-\Delta$,
 one has
\begin{equation*}
-i\partial_j (\sigma \cdot \bp)^{-1} 
=\{-i\partial_j/\sqrt{-\Delta}\} \{(\sigma \cdot \bp) /\sqrt{-\Delta}\}
=\sum_{k=1}^3\sigma_k R_j R_k
\end{equation*}
where
$R_j=-i\partial_j/\sqrt{-\Delta}$, $j=1$, $2$, $3$,
are the  Riesz transforms. 
Since $R_j$ is a pseudo-differential operator with 
simbol $\xi_j/|\xi|$,
it is bounded on $L^p$
by the Calder\'on-Zygmund theorem (see \cite{Stein}).
The first inequality in (\ref{eqn:IchiSait})
therefore follows.
}
\vspace{15pt}

\end{document}